\documentclass{amsart}
\usepackage{hyperref}
\usepackage{amssymb}
\usepackage{ifthen}
\usepackage{graphicx}

\newtheorem{thm}{Theorem}
\newtheorem{cor}{Corollary}
\newtheorem{lem}{Lemma}

\newtheorem{rem}{Remark}

\newtheorem{conj}{Conjecture}
\theoremstyle{definition}
\newtheorem{example}[equation]{Example}
\newtheorem{prob}[equation]{Problem}

\newcommand{\A}{{\mathcal A}}

\newcommand{\U}{{\mathcal U}}
\newcommand{\es}{{\mathcal S}}

\newcommand{\D}{{\mathbb D}}

\def\be{\begin{equation}}
\def\ee{\end{equation}}

\newcommand{\bee}{\begin{enumerate}}
\newcommand{\eee}{\end{enumerate}}

\newcommand{\blem}{\begin{lem}}
\newcommand{\elem}{\end{lem}}
\newcommand{\bthm}{\begin{thm}}
\newcommand{\ethm}{\end{thm}}
\newcommand{\bcor}{\begin{cor}}
\newcommand{\ecor}{\end{cor}}
\newcommand{\beg}{\begin{example}}
\newcommand{\eeg}{\end{example}}
\newcommand{\begs}{\begin{examples}}
\newcommand{\eegs}{\end{examples}}
\newcommand{\bdefe}{\begin{defin}}
\newcommand{\edefe}{\end{defin}}
\newcommand{\bprob}{\begin{prob}}
\newcommand{\eprob}{\end{prob}}
\newcommand{\bei}{\begin{itemize}}
\newcommand{\eei}{\end{itemize}}

\newcommand{\bcon}{\begin{conj}}
\newcommand{\econ}{\end{conj}}
\newcommand{\bcons}{\begin{conjs}}
\newcommand{\econs}{\end{conjs}}
\newcommand{\bprop}{\begin{propo}}
\newcommand{\eprop}{\end{propo}}
\newcommand{\br}{\begin{rem}}
\newcommand{\er}{\end{rem}}
\newcommand{\brs}{\begin{rems}}
\newcommand{\ers}{\end{rems}}
\newcommand{\bo}{\begin{obser}}
\newcommand{\eo}{\end{obser}}
\newcommand{\bos}{\begin{obsers}}
\newcommand{\eos}{\end{obsers}}
\newcommand{\bpf}{\begin{pf}}
\newcommand{\epf}{\end{pf}}
\newcommand{\ba}{\begin{array}}
\newcommand{\ea}{\end{array}}
\newcommand{\beq}{\begin{eqnarray}}
\newcommand{\beqq}{\begin{eqnarray*}}
\newcommand{\eeq}{\end{eqnarray}}
\newcommand{\eeqq}{\end{eqnarray*}}

\begin{document}
\bibliographystyle{amsplain}

\title[Certain properties of univalent functions with real coefficients]{Certain properties of the class of univalent functions with real coefficients}

\author[M. Obradovi\'{c}]{Milutin Obradovi\'{c}}
\address{Department of Mathematics,
Faculty of Civil Engineering, University of Belgrade,
Bulevar Kralja Aleksandra 73, 11000, Belgrade, Serbia}
\email{obrad@grf.bg.ac.rs}

\author[N. Tuneski]{Nikola Tuneski}
\address{Department of Mathematics and Informatics, Faculty of Mechanical Engineering, Ss. Cyril and Methodius
University in Skopje, Karpo\v{s} II b.b., 1000 Skopje, Republic of Macedonia.}
\email{nikola.tuneski@mf.edu.mk}


\maketitle

\begin{abstract}
Let $\U^+$ be the class of analytic functions  $f$ such that $\frac{z}{f(z)}$ has real and positive coefficients and $f^{-1}$ be its inverse. In this paper we give sharp estimates of the initial coefficients and initial logarithmic coefficients for $f$, as well as, sharp estimates of the second  and the third Hankel determinant for $f$ and $f^{-1}$. We also show that the Zalcman conjecture holds for functions $f$ from $\U^+$.
\end{abstract}


\section{Introduction}
Let $\mathcal{A}$ be the class of functions $f$ that are analytic  in the open unit disc $\D=\{z:|z|<1\}$ of the form
\begin{equation}\label{e1}
f(z)=z+a_2z^2+a_3z^3+\cdots,
\end{equation}
let $\mathcal{S}$ be its subclass consisting of univalent functions from $\A$, and $\mathcal{S}^+$ consists of functions $f$ from $\es$ with representation
\begin{equation}\label{e3}
\frac{z}{f(z)} = 1+b_1z+b_2z^2+\cdots,\qquad b_n\ge0, \quad n=1,2,3,\ldots.
\end{equation}
Note that the Silverman class of univalent functions with negative coefficients, i.e., with  expansion
\[  f(z) = z-a_2z^2-a_3z^3+\cdots,\qquad a_n\ge0, \quad n=1,2,3,\ldots, \]
is subclass of $\es^+$ since $z/f(z)$ satisfies \eqref{e3}. Also, the Koebe function $k(z)=\frac{z}{(1+z)^2}$ is in $\es^+$.

\medskip

Further, with $\U(\lambda)$, $0\le\lambda<1$, we will denote the class of functions $f$ from $\A$ satisfying the condition
\begin{equation}\label{e2}
\left| \left( \frac{z}{f(z)} \right)^2f'(z)-1 \right|<\lambda,\quad z\in\D,
\end{equation}
while $\U\equiv \U(1)$. This class of functions attracts significant interest in the past decades. The class is intriguing because it doesn't follow the usual pattern to embed or be embedded in the class of starlike functions (functions that map the unit disk onto a starlike region). A collection of the more significant once can be found in \cite[Chapter 12]{TTV}.

\medskip

If we denote with $\U^+(\lambda)$ the class of functions that satisfy \eqref{e3} and \eqref{e2}, and additionally $\U^+\equiv\U^+(1)$, then \cite{OP_2009} we have the following equivalence
\[ f\in\es^+ \quad \Leftrightarrow \quad f\in\U^+\quad \Leftrightarrow \quad  \sum_{n=2}^\infty (n-1)b_n\le1. \]

\medskip

Next lemma makes an extension of the second equivalence.

\medskip

\begin{lem}\label{lem-1}
$f\in\U^+(\lambda) \Leftrightarrow   \sum_{n=2}^\infty (n-1)b_n\le\lambda.$
\end{lem}

\begin{proof}
If $f\in\U^+(\lambda) $, then by definition
\[
\begin{split}
\quad \left| \left( \frac{z}{f(z)} \right)^2f'(z)-1 \right| = \left| \frac{z}{f(z)} - z\left( \frac{z}{f(z)} \right)' -1 \right| = \left|-\sum_{n=2}^\infty (n-1)b_nz^n \right| <\lambda,
\end{split}
\]
$z\in\D.$ For real $z$ and $z\rightarrow 1$ from left, the last inequality gives $\sum_{n=2}^\infty (n-1)b_n\le\lambda.$

\medskip

If $\sum_{n=2}^\infty (n-1)b_n\le\lambda$, then
\[ \left| \left( \frac{z}{f(z)} \right)^2f'(z)-1 \right| =\left|-\sum_{n=2}^\infty (n-1)b_nz^n \right|
\le \sum_{n=2}^\infty (n-1)b_n|z|^n <  \sum_{n=2}^\infty (n-1)b_n \le \lambda,\]
$z\in\D$, which mens that $f\in\U^+(\lambda) $.
\end{proof}

\medskip

In \cite{OT-2020-1} the authors gave sharp bounds of the first five coefficients in the expansion of $f$. The result for the case $\lambda=1$ reduces to the following.

\medskip

\noindent
{\bf Theorem A.}
{\it
Let $f(z)=z+a_2z^2+a_3z^3+\cdots \in\es^+=\U^+$. Then the following estimates are sharp
\[ -2\le a_2\le0, \quad -1\le a_3\le3, \quad -4\le a_4\le\frac43\sqrt{\frac23}, \quad -\frac94\le a_5\le5. \]
}

\medskip

Let note that for the general class $\U$ we have $|a_n|\le n$, $n=2,3,\ldots$, independently from the de Branges theorem. Namely, for the functions $f$ from $\U$, we have
\[ \frac{f(z)}{z}\prec \frac{1}{(1-z)^2} = 1+\sum_{n=1}^\infty nz^{n-1}, \]
and the rest follows from the Rogosinski theorem (\cite[Theorem 3.2.9, p.35]{TTV}).

\medskip

Further, in \cite{OP_2016}, the authors proved that for $f$ in the general class $\U(\lambda)$, $0<\lambda\le1$,  $|a_2|\le 1+\lambda$, and if $|a_2|=1+\lambda$, then $f$ must be of the form
\[ f_\theta (z) = \frac{z}{1-(1+\lambda)e^{i\theta}z+\lambda e^{2i\theta}z^2}. \]
For the case $f\in\U^+(\lambda)$ we have the following corresponding result.

\medskip

\noindent
{\bf Theorem B.}
{\it
If $f$ is given by \eqref{e1} with \eqref{e3}, and $f\in\U^+(\lambda)$, $0<\lambda\le1$, then $-(1+\lambda)\le a_2\le0$ (or $0\le b_1\le1+\lambda$). Moreover, if $a_2=-(1+\lambda)$, then $f$ must be of the form
\begin{equation}\label{e5}
f_\lambda(z) = \frac{z}{1+(1+\lambda)z+\lambda z^2},
\end{equation}
i.e., $b_2=\lambda$, and $b_3=b_4=\cdots=0$,
}

\medskip

In this paper we study coefficient problems for inverse functions of functions in $\U^+(\lambda)$, more precisely we will give sharp upper bounds of the leading coefficients and leading logarithmic coefficients of the inverse functions, as well as estimates of the modulus of the second and the third Hankel determinant.

\section{Coefficient estimates}

The famous Koebe 1/4 theorem guaranties that each function $f$ from $\es$ has an inverse at least on a disk with radius 1/4. Let the inverse function has an expansion
\begin{equation}\label{e6}
f^{-1}(w) = w+A_2w^2+A_3w^3+\cdots
\end{equation}
on the disk $|w|<\frac14$. By using the identity $f(f^{-1}(w))=w$, and representations \eqref{e1} and \eqref{e6}, we obtain the following relations:
\begin{equation}\label{e7}
\begin{array}{l}
A_{2}=-a_{2}, \\[2mm]
A _{3}=-a_{3}+2a_{2}^{2} , \\[2mm]
A_{4}= -a_{4}+5a_{2}a_{3}-5a_{2}^{3}.
\end{array}
\end{equation}

\medskip

We now give sharp bounds of $A_2$, $A_3$ and $A_4$.

\medskip

\begin{thm}\label{th-1}
Let $f\in \U^+(\lambda)$, $0<\lambda\le1$, and $f^{-1}$ is given by \eqref{e6}. Then, the following estimates are sharp:
\begin{equation}\label{e8}
\begin{array}{l}
0\le A_{2} \le 1+\lambda, \\[2mm]
0\le A _{3}\le 1+3\lambda+\lambda^2 , \\[2mm]
0\le A_{4} \le (1+\lambda)(1+5\lambda+\lambda^2).
\end{array}
\end{equation}
\end{thm}

\begin{proof}
Since $f\in \U^+(\lambda)$, $0<\lambda\le1$, then by definition
\[  f(z) = z+a_2z^2+a_3z^3+\cdots=\frac{z}{1+b_1z+b_2z^2+\cdots},\]
where $b_n\ge0, \quad n=1,2,3,\ldots.$ After comparing the coefficients we receive,
\begin{equation}\label{e9}
\begin{array}{l}
a_{2} = -b_1, \\[2mm]
a _{3} = -b_2+b_1^2, \\[2mm]
a_{4} = -b_3+2b_1b_2-b_1^3,\\[2mm]
a_{5} = -b_4+b_2^2+2b_1b_3-3b_1^2b_2+b_1^4,
\end{array}
\end{equation}
and by combining \eqref{e9} with \eqref{e7},
\begin{equation}\label{e12}
\begin{array}{l}
A_{2}=b_1   , \\[2mm]
A _{3}= b_2+b_1^2 , \\[2mm]
A_{4}= b_3+3b_1b_2+b_1^3.
\end{array}
\end{equation}

\medskip

Since $a_2$ is real (so are $a_3$, $a_4$,\ldots), and $|a_2|\le1+\lambda$ for all functions from $\U(\lambda)$ (see \cite{vasu-2013} or \cite[Theorem 12.3.1, p.188]{TTV}), we receive $-(1+\lambda)\le a_2\le1+\lambda$.

\medskip

From $A_2=-a_2=b_1$, $b_1\ge0$ and $a_2\le1+\lambda$, easily follows the estimate
\begin{equation}\label{e10}
0\le A_2=b_1\le1+\lambda.
\end{equation}

\medskip

Further, from Lemma \ref{lem-1} we have
\[ b_2+2b_3+3b_4+\cdots \le\lambda,  \]
which implies
\begin{equation}\label{e11}
  \begin{split}
    0 \le b_2 &\le \lambda \\
    b_2+2b_3 &\le \lambda \quad \left( \Leftrightarrow b_3\le \frac12(\lambda-b_2) \right)\\
    b_2+2b_3+3b_4 &\le \lambda \quad \left( \Leftrightarrow b_4\le \frac13(\lambda-2b_3-b_2) \right).
  \end{split}
\end{equation}

\medskip

Combining \eqref{e10}, \eqref{e11} and \eqref{e12}, we receive
\[ 0\le A_3 \le \lambda+(1+\lambda)^2 = 1+3\lambda+\lambda^2,\]
and for $A_4$,
\[
\begin{split}
0 \le A_4 &\le \frac12(\lambda-b_2)+3(1+\lambda)b_2 +(1+\lambda)^3\\
&= \frac{\lambda}{2}+\left(\frac52+3\lambda\right)b_2+(1+\lambda)^3\\
&\leq \frac{\lambda}{2}+\left(\frac52+3\lambda\right)\lambda+(1+\lambda)^3\\
&= (1+\lambda)(1+5\lambda+\lambda^2).
\end{split}
\]
\medskip

For the sharpness of the upper bounds, for the function $f_\lambda(z)$ defined by \eqref{e5}, we have
\[\begin{split}
  f_\lambda(z) &= \frac{z}{1+(1+\lambda)z+\lambda z^2} \\
  &= z-(1+\lambda)z^2+(1+\lambda+\lambda^2)z^3 - (1+\lambda+\lambda^2+\lambda^3)z^4+\cdots,
  \end{split}
\]
i.e.,  $f_{\lambda}\in\U^+(\lambda)$ and
\[ z = f_{\lambda}^{-1}(w) = w+(1+\lambda)w^2 + (1+3\lambda+\lambda^2)w^3+(1+\lambda)(1+5\lambda+\lambda^2)w^4+\cdots. \]

\medskip

For the sharpness of the lower bounds it is enough to consider the inverse function of the function
\[ f(z) = \frac{z}{1+\lambda z^4/3},\]
where $b_1=b_2=b_3=0$.
\end{proof}

\medskip

\begin{rem}
In \cite{OT-2021-2}, the authors obtained the same result, but with an additional condition
\begin{equation}\label{cond}
\frac{f(z)}{z}\prec\frac{1}{(1+z)(1+\lambda z)}.
\end{equation}
In our proof we didn't use that condition.
\end{rem}

\begin{rem}\label{rem-2}
From \eqref{e9} we have $a_3=b_1^2-b_2\equiv \psi_1(b_1)$, $0\le b_1\le 1+\lambda$. It is evident that $\psi_1(b_1)\ge -b_2\ge-\lambda$, while $\max\psi_1=\psi_1(1+\lambda)$. But, $b_1=1+\lambda$ implies $b_2=\lambda$, so
\[\psi_1(1+\lambda)=(1+\lambda)^2-\lambda=1+\lambda+\lambda^2.\]
Similarly, by \eqref{e9} we have
\[ -a_4=b_1^3-2b_1b_2+b_3\equiv \psi_2(b_1),\]
$0\le b_1\le1+\lambda$, and by using the previous comments,
\[ \max\psi_2 = \psi_2(1+\lambda) = (1+\lambda)^3-2(1+\lambda)\lambda = 1+\lambda+\lambda^2+\lambda^3.\]

\medskip

It means that $a_4\ge -(1+\lambda+\lambda^2+\lambda^3)$. If we use the upper bound for $a_4$ given in \cite{OT-2020-1}, then
\[ |a_4|\le 1+\lambda+\lambda^2+\lambda^3 \]
for $f\in\U^+(\lambda)$.
It means that the results given in Theorem 3 from \cite{OT-2020-1} are valid independently of the condition \eqref{cond}.

\medskip

This is the reason we pose the next conjecture.
\end{rem}

\begin{conj}
For the functions in the class $\U^+(\lambda)$,
\[ |a_n|\le \frac{1-\lambda^n}{1-\lambda} = 1+\lambda+\lambda^2+\lambda^3+\cdots+\lambda^n, \quad n=2,3,4,\ldots.\]
This was proven to hold for the general class $\U(\lambda)$ under the condition \eqref{cond} in the cases $n=2,3,4$ (\cite{mont} or \cite[Theorem 12.3.2, p.191]{TTV}.
\end{conj}

\medskip

In \cite{OT-2020-1} the authors studied the logarithmic coefficients for the class $\U^+$ ($\es^+$), and here we will continue and give estimates for the logarithmic coefficients of the inverse functions of functions in the classes $\U^+(\lambda)$. For the logarithmic coefficients ($\Gamma_1$, $\Gamma_2$, $\Gamma_3$,\ldots) we have
\[\log \frac{f^{-1}(w)}{w}=\log(1+A_{2}w+A_{3}w^{2}+\cdots) ,\]
or
\[\sum _{n=1}^{\infty}2\Gamma_{n}w^{n}=A_{2}w + \left(A_3-\frac{1}{2}A_2^{2}\right)w^{2}
+\left(A_4-A_2A_3+\frac13A_2^3\right)w^{3}+\cdots  .\]
From the last relation we get
\begin{equation}\label{e13}
\begin{array}{l}
\Gamma_1 = \frac12A_2 = \frac12 b_1,\\[2mm]
\Gamma_2 = \frac12\left(A_3-\frac12A_2^2 \right) = \frac12\left(b_2+\frac12b_1^2 \right) ,\\[2mm]
\Gamma_3 = \frac12\left( A_4-A_2A_3+\frac13A_2^3 \right) = \frac12\left( b_3+2b_1b_2+\frac13b_1^3 \right) ,
\end{array}
\end{equation}
where we used relations \eqref{e12}.

\medskip

\begin{thm}
Let $f\in\U^+(\lambda)$, $0< \lambda\le1$. Then:
\begin{itemize}
  \item[(a)] $0\le\Gamma_1\le \frac{1+\lambda}{2}$;
  \item[(b)] $0\le\Gamma_2\le \frac14(1+4\lambda+\lambda^2)$;
  \item[(c)] $0\le\Gamma_3\le \frac16(1+\lambda)(1+8\lambda+\lambda^2)$,
\end{itemize}
and all results are sharp.
\end{thm}

\begin{proof}
The proof is similar to the one of Theorem \ref{th-1}.
\end{proof}

\medskip

\begin{cor}
For $\lambda=1$ we receive sharp estimates for the inverse function of functions from $\U^+(1)=\U^+=\es^+$:
\[ 0\le\Gamma_1\le1,\quad    0\le\Gamma_2\le\frac32,\quad  0\le\Gamma_3\le\frac{10}{3}.\]
\end{cor}

\medskip

\section{The second and the third Hankel determinant}

Another concept, rediscovered few years ago, attracting the  attention of mathematicians working in the field of univalent functions is so called Hankel determinant $H_{q}(n)(f)$ of a given function $f(z)=z+a_2z^2+a_3z^3+\cdots,$, for $q\geq 1$ and $n\geq 1$, defined by
\[
        H_{q}(n)(f) = \left |
        \begin{array}{cccc}
        a_{n} & a_{n+1}& \ldots& a_{n+q-1}\\
        a_{n+1}&a_{n+2}& \ldots& a_{n+q}\\
        \vdots&\vdots&~&\vdots \\
        a_{n+q-1}& a_{n+q}&\ldots&a_{n+2q-2}\\
        \end{array}
        \right |.
\]
Main interest is to find upper bound (preferably sharp) of the modulus of $H_{q}(n)(f)$. The general Hankel determinant is hard to deal with, so the second and the third ones,
\[H_{2}(2)(f)= \left|\begin{array}{cc}
        a_2& a_3\\
        a_3& a_4
        \end{array}
        \right | = a_2a_4-a_{3}^2\] and
\[ H_3(1)(f) =  \left |
        \begin{array}{ccc}
        1 & a_2& a_3\\
        a_2 & a_3& a_4\\
        a_3 & a_4& a_5\\
        \end{array}
        \right | = a_3(a_2a_4-a_{3}^2)-a_4(a_4-a_2a_3)+a_5(a_3-a_2^2),
\]
respectively, are studied instead. The research is focused on the subclasses of univalent functions (starlike, convex, $\alpha$-convex, close-to-convex, spirallike,\ldots) since the general class of normalised univalent functions is also hard to deal with. Some of the more significant results can be found in \cite{ckkls1,ckkls2,hayman-68,jhd1,jhd2,Kowalczyk-18,lss,kls,lrs,MONT-2019-2,zaprawa}.

\medskip

In this section we will give sharp bounds of the second  and the third Hankel determinant for the functions in $\U^+(\lambda)$ and for their inverse.

\medskip

Using \eqref{e9}, \eqref{e10} and \eqref{e11}, after some calculations, we receive
\begin{equation}\label{e14}
\begin{split}
  H_2(2)(f) &= b_1b_3-b_2^2,\\
  H_3(1)(f) &= b_2b_4-b_3^2.
  \end{split}
  \end{equation}

\medskip

\begin{thm}
Let  $f\in\U^+(\lambda)$, $0<\lambda\le1$. Then the following estimates are sharp:
\begin{itemize}
  \item[($i$)] $-\lambda^2 \le H_2(2)(f) \le \left(1-\frac{\lambda}{2} \right)\frac{\lambda}{2}$;
  \item[($ii$)] $-\frac{\lambda^2}{4} \le H_3(1)(f) \le \frac{\lambda^2}{12}$.
\end{itemize}
\end{thm}

\begin{proof}$ $
\begin{itemize}
  \item[($i$)]
First, let note that from \eqref{e10} and \eqref{e11}, we have  $b_1,b_3\ge0$ and $b_2\le\lambda$, which easily implies
\[ H_2(2)(f) = b_1b_3-b_2^2 \ge -b_2^2\ge-\lambda^2. \]
The estimate is sharp due to the function $f(z)=\frac{z}{1+\lambda z^2}$.

  \medskip

For the upper bound, from \eqref{e14} we have $H_2(2)(f) = b_1b_3-b_2^2 \le b_1b_3$, and further from $0\le b_1\le 1+\lambda$ and $0\le b_3\le \frac{\lambda}{2}$ (see \eqref{e10} and \eqref{e11}) we need to find $\max\{b_1b_3\}$. For $b_3=\frac{\lambda}{2}$ ($\Rightarrow$ $b_2=b_4=\cdots=0$) we have
\[ H_2(2)(f)\le \frac{\lambda}{2} b_1.\]
If we consider the function $f$ defined by
\[ \frac{z}{f(z)} = 1+b_1z+\frac{\lambda}{2}z^3,\]
then $f$ is in $\U^+(\lambda)$ if $1+b_1z+\frac{\lambda}{2}z^3\neq 0$. Since
\[ \left.\frac{z}{f(z)} \right|_{z=1} = 1+b_1+\frac{\lambda}{2}>0 \quad \text{and} \quad \left.\frac{z}{f(z)} \right|_{z=-1} = 1-b_1-\frac{\lambda}{2}, \]
then it is necessary  $1-b_1-\frac{\lambda}{2}\ge0$. In contrary, if $1-b_1-\frac{\lambda}{2}<0$, then for $z=r$ real in $-1<r<1$, we have
\[ \left.\frac{z}{f(z)} \right|_{z=r} = 1+b_1r+\frac{\lambda}{2}r^3,  \]
has zero in $(-1,1)$, i.e.,   $f$ is not analytic on the unit disk.

\medskip

The previous condition is also a sufficient one. Namely, for $0\le b_1\le 1-\frac{\lambda}{2}$ and $0\le b_3\le\frac{\lambda}{2}$:
\[
\begin{split}
\left| \frac{z}{f(z)} \right| &= \left|1+b_1z+\frac{\lambda}{2}z^3 \right| \ge 1 - b_1|z| - \frac{\lambda}{2}|z|^3\\
& > 1 - \left(1-\frac{\lambda}{2}\right) - \frac{\lambda}{2}=0.
\end{split}
\]
So, $0\le b_1\le 1-\frac{\lambda}{2}$ and $H_2(2)(f) \le \left(1-\frac{\lambda}{2} \right)\frac{\lambda}{2}$. The upper bound is sharp due to the function $f(z)=\frac{z}{1+(1-\lambda/2)z +\lambda/2 z^3}$.

\medskip

\item[($ii$)] Similarly as in the proof of part (a), from \eqref{e14} and \eqref{e11},
\[
\begin{split}
H_3(1)(f) &\le b_2\cdot \frac13(\lambda-b_2-2b_3)-b_3^2 = \frac13\lambda b_2 - \frac13 b_2^2 - \frac23 b_2b_3-b_3^2\\
& \le \frac13\lambda b_2 - \frac13 b_2^2\le \frac{\lambda^2}{12}    .
\end{split}
\]
Also, $H_3(1)(f) \ge -b_3^2\ge -\left(\frac{\lambda}{2}\right)^2 = - \frac{\lambda^2}{4}$, since by \eqref{e11}, $b_3\le\frac{\lambda-b_2}{2}\le \frac{\lambda}{2}.$ The functions $f(z)=\frac{z}{1+\lambda/2z^3}$ and $f(z)=\frac{z}{1+\lambda/2z^2+\lambda/6z^4}$ show that the estimates are sharp.
\end{itemize}
\end{proof}

\medskip

For the sharp estimates of the Hankel determinant for the inverse of the functions from $\U^+(\lambda)$ we have the following theorem.

\medskip

\begin{thm}
Let  $f\in\U^+(\lambda)$, $0<\lambda\le1$. Then the following estimates are sharp:
\begin{itemize}
  \item[($i$)] $-\lambda^2 \le H_2(2)(f^{-1}) \le \lambda(1+\lambda+\lambda^2)$;
  \item[($ii$)] $-\frac{\lambda^2}{4} \le H_3(1)(f^{-1}) \le \lambda^3$.
\end{itemize}
\end{thm}

\begin{proof}$ $
\begin{itemize}
  \item[($i$)] From \eqref{e12} we have
  \[ H_2(2)(f^{-1}) = A_2A_4-A_3^2 = b_1b_3+b_1^2b_2-b_2^2 \equiv \psi_1(b_1), \]
  where $\psi_1$ is an increasing function of $b_1$, $0\le b_1\le1+\lambda$, and $\psi_1(b_1)\le\psi(1+\lambda)$. But, Lemma \ref{lem-1}, implies $b_2=\lambda$ and $b_3=0$ when for $b_1=1+\lambda$. So, $\psi_1(b_1)\le\psi(1+\lambda)=(1+\lambda)\cdot0+(1+\lambda)^2\lambda-\lambda^2=\lambda(1+\lambda+\lambda^2)$, and the result follows.
      This upper bound is sharp due to the function $f_\lambda$ defined in \eqref{e5}.

    \medskip

    Also, $H_2(2)(f^{-1})  \ge -b_2^2 \ge -\lambda^2$. This bound is also sharp as the function $f(z)=\frac{z}{1+\lambda z^2}$ shows.

  \medskip
  \item[($ii$)] For the third Hankel determinant of $f^{-1}$ we have
  \[ H_3(1)(f^{-1}) = H_3(1)(f) -(a_3-a_2^2)^3 = b_2b_4-b_3^2+b_2^3,    \]
   which is an increasing function of $b_2$ on the interval $(0,\lambda)$, reaching the maximum value on the interval for $b_2=\lambda$. At the same moment, due to Lemma \ref{lem-1} we receive that $b_3=b_4=0$, leading to
   \[ H_3(1)(f^{-1}) \le \lambda\cdot 0 - 0^2+\lambda^3 = \lambda^3.\]
   The function $\frac{z}{1+\lambda z^3}$ shows that this estimate is the best possible.

   \medskip

   Also, $H_3(1)(f^{-1}) \ge -b_3^2 \ge -\frac{\lambda^2}{4}$ since from Lemma \ref{lem-1}, under the condition  $b_n=0$ for all positive integers $n\neq 3$, we receive $b_3\le\frac{\lambda}{2}$. This estimate is also the best possible with an extremal function $\frac{z}{1+\lambda/2 z^3}$.
\end{itemize}
\end{proof}

\medskip

\section{Zalcman conjecture}

In the early 1970's Zalcman posed the following conjecture for the class of univalent functions:
\[ |a_n^2-a_{2n-1}|\le (n-1)^2 \quad\quad (n\in {\mathbb N}, n\ge2).\]
There is a manuscript by Krushkal (\cite{krus})  that uses complex geometry of the universal Teichm\"{u}ller space claiming to have proven the conjecture, but this work is not widely reckognized to be correct.
Later, in 1999, Ma (\cite{ma}) proposed a generalized Zalcman conjecture,
\[ |a_m a_n-a_{m+n-1}|\le (m-1)(n-1) \quad\quad (m,n\in {\mathbb N}, m\ge2, n\ge2),\]
and closed it  for the class of starlike functions and for the class of univalent functions with real coefficients. The general case is still an open problem.

\medskip

The Zalcman conjecture for the class ${\mathcal{U}}$ when $n=2$ and $n=3$ was proven in \cite{MONT-NSJOM}.

\medskip

\noindent
{\bf Theorem C.}
{\it
Let $f\in{\mathcal{U}}$ be of the form \eqref{e1}. Then
\begin{itemize}
  \item[$(i)$] $|a_2a_3-a_4|\le2$;
  \item[$(ii)$] $|a_2a_4-a_5|\le3$.
\end{itemize}
These inequalities are sharp with equality for the Koebe function $k(z)=\frac{z}{(1-z)^2}=z+\sum_{n=2}nz^n$ and its rotations.
}

\medskip

We now give direct proof of the Zalcman conjecture for the class ${\mathcal{U}^+(\lambda)}$ for the cases when  $n=2$ and $n=3$.

\medskip

\begin{thm}
Let $f\in\mathcal{U}^+(\lambda)$, $0<\lambda\le1$, be of the form \eqref{e1}. Then
\begin{itemize}
  \item[$(i)$] $-(1+\lambda)\lambda \le a_2a_3-a_4\le \frac12 \lambda$;
  \item[$(ii)$] $|a_2a_4-a_5|\le \lambda+\lambda^2+\lambda^3$.
\end{itemize}
These inequalities are sharp.
\end{thm}

\medskip

\begin{proof}$ $
\begin{itemize}
  \item[($i$)] Using \eqref{e9} we have
  \[ a_2a_3-a_4 = -b_1b_2+b_3 \le b_3 \le\frac12\lambda, \]
  (where we used Lemma \ref{lem-1} or \eqref{e11}).   Also,
  \[ -(a_2a_3-a_4) = b_1b_2-b_3 \le b_3 \le(1+\lambda)\lambda, \]
  (see Remark \ref{rem-2}), i.e., $a_2a_3-a_4\le(1+\lambda)\lambda$.
  The functions $f_\lambda$ and $f(z)=\frac{z}{1+\frac{\lambda}{2}z^3}$ show that the result is sharp.

  \medskip

  \item[($ii$)] Using \eqref{e9} we obtain that
  \[ a_2a_4-a_5 = b_1^2b_2-b_1b_3-b_2^2+b_4 \equiv \psi_2(b_1),\]
  $0\le b_1\le1+\lambda$. Since, $\max\psi_2=\psi_2(1+\lambda) = (1+\lambda)^2\lambda-\lambda^2 = \lambda+\lambda^2+\lambda^3$, then
  \[a_2a_4-a_5 \le \lambda+\lambda^2+\lambda^3.\]
   On the other hand,
  \[a_2a_4-a_5 \ge -b_1b_3-b_2^2 \ge   -(1+\lambda)\frac{\lambda}{2}-\lambda^2 > -(\lambda+\lambda^2+\lambda^3).\]
    The result is sharp due to the function $f_\lambda$.
\end{itemize}
\end{proof}

\end{document}